\documentclass[reqno]{amsart}%
\usepackage{amsmath}
\usepackage{amsfonts}
\usepackage{CJK}
\usepackage{amssymb}
\usepackage{graphicx}%
\usepackage{hyperref}
\usepackage{amscd}
\usepackage{mathrsfs}
\usepackage{bbm}

\newtheorem{thm}{Theorem}[section]

\newtheorem{lem}[thm]{Lemma}

\theoremstyle{definition}
\newtheorem{defn}{Definition}[section]
\theoremstyle{Conjecture}

\theoremstyle{remark}
\newtheorem{rem}{Remark}[section]
\theoremstyle{Example}

\newcommand{\be}{\begin{equation}}
\newcommand{\ee}{\end{equation}}
\newcommand{\bea}{\begin{eqnarray}}
\newcommand{\eea}{\end{eqnarray}}
\newcommand{\ben}{\begin{eqnarray*}}
\newcommand{\een}{\end{eqnarray*}}
\newcommand{\bet}{\begin{equation}
\begin{split}}
\newcommand{\eet}{\end{split}
\end{equation}}

\begin{document}
\title[Cluster points of jumping coefficients of plurisubharmonic functions]
{Cluster points of jumping coefficients and equisingularties of plurisubharmonic functions}

\author{Qi'an Guan}
\address{Qi'an Guan: School of Mathematical Sciences, and Beijing International Center for Mathematical Research,
Peking University, Beijing, 100871, China.}
\email{guanqian@amss.ac.cn}
\author{Zhenqian Li}
\address{Zhenqian Li: School of Mathematical Sciences, Peking University, Beijing, 100871, China.}
\email{lizhenqian@amss.ac.cn}

\thanks{The first author was partially supported by NSFC-11522101 and NSFC-11431013.}

\date{\today}

\begin{abstract}
In this article, we will construct a plurisubharmonic function whose jumping coefficients have a cluster point. We also give a class of plurisubharmonic functions which cannot be equisingular to any plurisubharmonic function with generalized analytic singularities.
\end{abstract}
\maketitle

\section{Introduction}\label{sec:introduction}

\subsection{Cluster points of jumping coefficients}
Let $X$ be a complex manifold of dimension $n$ and $\varphi$ a quasi-plurisubharmonic (abbr. quasi-psh) function on $X$. The \emph{multiplier ideal sheaf} $\mathscr{I}(\varphi)$ is defined to be the sheaf of germs of holomorphic functions $f$ such that $|f|^2e^{-2\varphi}$ is locally integrable, which is a coherent sheaf of ideals (see \cite{De10}).

Let $x\in X$ be a point. By the solution to Demailly's strong openness conjecture \cite{G-Z_open}, i.e., $$\mathscr{I}(\varphi)=\mathscr{I}_+(\varphi):=\bigcup_{\varepsilon>0}\mathscr{I}((1+\varepsilon)\varphi),$$
which was posed by Demailly in \cite{De01}, there is an increasing sequence $$0=\xi_0(\varphi;x)<\xi_1(\varphi;x)<\xi_2(\varphi;x)<\cdots$$ of real numbers $\xi_k=\xi_k(\varphi;x)$ such that $\mathscr{I}(c{\cdot}\varphi)_x=\mathscr{I}(\xi_k{\cdot}\varphi)_x$ for $c\in[\xi_k,\xi_{k+1})$ and $$\mathscr{I}(\xi_{k+1}{\cdot}\varphi)_x\subsetneqq\mathscr{I}(\xi_k{\cdot}\varphi)_x\quad \mbox{for\ every\ } k.$$
If the Lelong number $\nu(\varphi,x)=0$, we put $\xi_1=+\infty$.

\begin{defn} (\cite{ELSV}, \cite{De15}).
The real numbers $\xi_k(\varphi;x)$ is called the \emph{jumping coefficients} of \emph{jumping numbers} of $\varphi$ at $x$. We say that $\xi$ is a jumping coefficient of $\varphi$ on an analytic set $A\subset X$ if it is a jumping coefficient of $\varphi$ at some point $x\in A$. The collection of all jumping coefficients of $\varphi$ at $x$ is denoted by Jump($\varphi;x$).
\end{defn}

By the strong openness property of multiplier ideal sheaves, we know that Jump($\varphi;x$) satisfies the descending chain condition (DCC): any decreasing sequence of jumping coefficients stabilizes. Moreover, it is easy to see that Jump($\varphi;x$) does not satisfy ACC. By a discussion on Hironaka's log resolution \cite{De15}, the sequence $(\xi_k)$ will go to infinity when $\varphi$ has analytic singularities, i.e., $\varphi$ can be written locally as
$$\varphi=\frac{c}{2}\log\big(\sum\limits_{k=1}^{m}|f_{k}|^2\big)+O(1),$$
where $c\in\mathbb{R}^+$ and $f_{k}$ are holomorphic functions. However, the sequence $(\xi_k)$ will not go to infinity for general plurisubharmonic functions as we will present, which is very different from the algebraic case. One can refer to \cite{ELSV} and \cite{De15} for more properties of jumping coefficients.

In the present subsection, we will construct a plurisubharmonic function $\varphi$ such that Jump($\varphi;x$) has a cluster point, i.e., there exists a jumping coefficient $\xi_{k_0}$ such that for every small enough $\varepsilon>0$ we have a jumping coefficient $\xi_{k}\in(\xi_{k_0}-\varepsilon,\xi_{k_0})$. Indeed, we will prove the following:

\begin{thm} \label{cluster}
There exists a plurisubharmonic function $\varphi$ on $\mathbb{C}^n\ (n\geq2)$ such that $1$ is a cluster point of  \emph{Jump($\varphi;o$)}.
\end{thm}

By Theorem \ref{cluster}, one can obtain that the plurisubharmonic function $\varphi$ will not be equisingular to any plurisubharmonic function with analytic singularities (here, ``equisingular" for $\varphi_1$ and $\varphi_2$ means that $e^{-2\varphi_1}-e^{-2\varphi_2}$ is locally integrable). We will consider general case in the next subsection.

\subsection{Equisingular problem on plurisubharmonic functions}

In \cite{De01} (see also \cite{De10}), Demailly proved that, for any given quasi-psh function $\varphi$ on compact Hermitian manifold $X$, one can choose equisingular quasi-psh functions $\varphi_{j}\ (j=1, 2, ...)$ on $X$, which are smooth outside the poles and decreasingly convergent to $\varphi$. In \cite{Guan}, Guan showed that the above result does not hold if $\varphi_{j}$ are supposed to have analytic singularities.

A plurisubharmonic function $\varphi$ will be said to have \emph{generalized analytic singularities} if for any point $x_0\in X$ there exists a ball $B_r(x_0)$ such that for every complex line $L\subset B_r(x_0)$, $\varphi|_L$ has analytic singularities. The set of plurisubharmonic functions on $X$ with generalized analytic singularities is denoted by $\mathcal{E}(X)$. In particular, $\mathcal{E}(X)$ contains any $\varphi\in Psh(X)$, which can be written locally as
$$\varphi=\log\big(\sum\limits_{j=1}^{N}\prod\limits_{k=1}^{m_j}|f_{jk}|^{\lambda_{jk}}\big)+O(1),$$
where $\lambda_{jk}\in\mathbb{R}^+$ and $f_{jk}$ are holomorphic functions.

Now, it is natural to put forward such a question:\\

\textbf{Question.} \emph{For any given $\varphi\in Psh(X)$, can one choose $\varphi_{A}\in\mathcal{E}(X)$ such that $\varphi_{A}$ is equisingular to $\varphi$?}\\

In this subsection, we establish the following result on equisingularties of plurisubharmonic functions.

\begin{thm} \label{equi_general}
Let $\chi$ be a convex increasing function on $\mathbb{R}$ with $\lim\limits_{t\to-\infty}\chi'(t)=C_0>0$. Let $\varphi$ be a plurisubharmonic function near $o\in\mathbb{C}^n$ with $C_0\cdot\nu(\varphi,o)\geq n$ and $\varphi_A\in\mathcal{E}(\mathbb{C}^n)$ has generalized analytic singularities. If $e^{-2\chi\circ\varphi}-e^{-2\varphi_A}$ is locally integrable near $o$, then $\lim\limits_{t\to-\infty}|\chi(t)-C_0t|<\infty$.
\end{thm}

\begin{rem}
$(1)$ By Theorem \ref{equi_general}, for any plurisubharmonic function $\varphi$ near $o\in\mathbb{C}^n$ with $\nu(\varphi,o)>0$, there exists a convex increasing function $\chi$ on $\mathbb{R}$ such that $e^{-2\chi\circ\varphi}-e^{-2\varphi_A}$ is not locally integrable near $o$.

$(2)$ By a similar discussion in \cite{Guan}, we have that, for any complex manifold $X$ (compact or noncompact) with $\dim X\geq2$ and $z_0\in X$, there exists a quasi-psh function $\varphi$ on $X$ such that for any plurisubharmonic function $\varphi_A$ with generalized analytic singularities near $z_0$, $e^{-2\varphi}-e^{-2\varphi_A}$ is not locally integrable near $z_0$.
\end{rem}

Note that in Theorem \ref{equi_general}, the Lelong number $\nu(\chi\circ\varphi,o)\geq n$ is necessary. For Lelong number near 1 case, following from the plurisubharmonic function $\varphi$ as in Theorem \ref{cluster} ($M$ large enough), we obtain

\begin{thm} \label{equisingular}
Let $\varphi\in Psh(\mathbb{C}^n)$ be the plurisubharmonic function as in Theorem \ref{cluster} and $\varphi_A$ a plurisubharmonic function with generalized analytic singularities near the origin $o$. Then, $e^{-2\varphi}-e^{-2\varphi_A}$ is not locally integrable near $o$.
\end{thm}

\section{Some known results}
To prove main results, the following special case of Ohsawa-Takegoshi $L^2$ extension theorem and Siu's decomposition theorem are necessary.

\begin{thm} \emph{(\cite{D-K}, Theorem 2.1).} \label{O-T}
Let $\Omega\subset\mathbb{C}^n$ be a bounded pseudoconvex domain, and let $H$ be an affine linear subspace of $\mathbb{C}^n$ of codimension $p\geq1$ given by an orthogonal system $s$ of affine linear equations $s_1=\cdots=s_p=0$. For every $\beta<p$, there exists a constant $C_{\beta,n,\Omega}$ depending only on $\beta,n$ and the diameter of $\Omega$, satisfying the following property. For every plurisubharmonic function $\varphi\in Psh(\Omega)$ and $f\in\mathcal{O}(\Omega\cap H)$ with $\int_{\Omega{\cap}H}|f|^2e^{-\varphi}d\lambda_H<+\infty$, there exists an extension $F\in\mathcal{O}(\Omega)$ of $f$ such that
$$\int_{\Omega}|F|^2|s|^{-2\beta}e^{-\varphi}d\lambda_n\leq C_{\beta,n,\Omega}\int_{\Omega{\cap}H}|f|^2e^{-\varphi}d\lambda_H,$$
where $d\lambda_n$ and $d\lambda_H$ are the Lebesgue volume elements in $\mathbb{C}^n$ and $H$ respectively.
\end{thm}

\begin{thm} \emph{(see \cite{De10}, Theorem 2.18).} \label{Siu_decomposition}
Let $T$ be a closed positive current of bidimension $(p, p)$. Then $T$ can be written as a convergent series of closed positive currents $$T=\sum_{k=1}^{+\infty}\lambda_k[A_k]+R,$$
where $[A_k]$ is a current of integration over an irreducible analytic set of dimension $p$, and $R$ is a residual current with the property that $\dim E_c(R)<p$ for every $c>0$. This decomposition is locally and globally unique: the sets $A_k$ are precisely the $p$-dimensional components occurring in the upperlevel sets $E_c(T)$, and $\lambda_k=\min_{x\in A_k}\nu(T,x)$ is the generic Lelong number of $T$ along $A_k$.
\end{thm}

The proof of Lemma 4.7 in \cite{G-Z_Lelong1} implies
\begin{lem} \label{Fubini}
Let $p:\Delta^2\to\Delta,\ z=(z_1,z_2)\mapsto z_1$ and $\varphi\in Psh(\Delta^2)$. For almost all $z_1\in\Delta$ (in the sense of the Lebesgue measure on $\Delta$), the level set of Lelong numbers of $\varphi|_{p^{-1}(z_1)}$ satisfies
$$\{z|\nu(\varphi|_{p^{-1}(z_1)},z)=0\}=(\{z|\nu(\varphi,z)=0\}\cap p^{-1}(z_1)).$$
\end{lem}

For the proof of Theorem \ref{equi_general}, we also need the following two Lemmas.

\begin{lem} \label{lem1}
Let $\varphi_1, \varphi_2$ be two subharmonic functions near $o\in\mathbb{C}$. If $e^{-2\varphi_1}-e^{-2\varphi_2}$ is locally integrable near $o$ and the Lelong number $\nu(\varphi_1,o)\geq1$, then $\nu(\varphi_1,o)=\nu(\varphi_2,o)$.
\end{lem}

\begin{proof}
As $e^{-2\varphi_1}-e^{-2\varphi_2}$ is locally integrable near $o$, it follows from $\nu(\varphi_1,o)\geq1$ that $\nu(\varphi_2,o)\geq1$.

Suppose that $\nu(\varphi_1,o)=\nu(\varphi_2,o)$ does not hold. Without loss of generality, we put $\nu(\varphi_1,o)>\nu(\varphi_2,o)$. Let $a\in(\nu(\varphi_2,o),\min\{\nu(\varphi_1,o),\nu(\varphi_2,o)+1\})$. Then, $\varphi_i-(a-1)\log|z|$ are subharmonic near $o$ and
$$\nu(\varphi_1-(a-1)\log|z|,o)>1,\ \nu(\varphi_2-(a-1)\log|z|,o)<1.$$
That is to say $e^{-2(\varphi_1-(a-1)\log|z|)}-e^{-2(\varphi_1-(a-1)\log|z|)}$ is not locally integrable near $o$.
However, the local integrability of $e^{-2\varphi_1}-e^{-2\varphi_2}$ near $o$ implies that
$$|z|^{2(a-1)}(e^{-2\varphi_1}-e^{-2\varphi_2})=e^{-2(\varphi_1-(a-1)\log|z|)}-e^{-2(\varphi_1-(a-1)\log|z|)}$$
is locally integrable near $o$, which is a contradiction.
\end{proof}

\begin{lem} \label{lem2}
Let $\varphi_A=C_1\log|z|+O(1)$ near $o\in\mathbb{C}$ and $\varphi$ a subharmonic function near $o\in\mathbb{C}$ with $\nu(\varphi,o)=C_2$, where $C_1, C_2$ are positive constants. Let $\chi$ be a convex increasing function on $\mathbb{R}$ such that $\lim\limits_{t\to-\infty}\chi'(t)=C_0>0$ with $C_0C_2\geq k$ for some $k\in\mathbb{N}^+$. If $|z|^{2(k-1)}(e^{-2\chi\circ\varphi}-e^{-2\varphi_A})$ is locally integrable near $o$, then $\lim\limits_{t\to-\infty}|\chi(t)-C_0t|<\infty$.
\end{lem}

\begin{proof}
As $\chi''\geq0$ and $\lim\limits_{t\to-\infty}\chi'(t)=C_0>0$, we have $\chi'(t)\geq C_0$ for any $t\in\mathbb{R}$. Then, $\chi(t)-C_0t$ is a convex increasing function on $\mathbb{R}$. Suppose that $\lim\limits_{t\to-\infty}|\chi(t)-C_0t|<\infty$ does not hold. Then $\lim\limits_{t\to-\infty}(\chi(t)-C_0t)=-\infty$

Since $\nu(\chi\circ\varphi,o)=C_0C_2\geq k$, then $\nu(\chi\circ\varphi-(k-1)\log|z|,o)\geq1$. It follows from the integrability of $|z|^{2(k-1)}(e^{-2\chi\circ\varphi}-e^{-2\varphi_A})$ near $o$ that $|z|^{-2(k-1)}e^{-2\varphi_A}$ is not locally integrable near $o$, which implies that $\varphi_A-(k-1)\log|z|$ is subharmonic near $o\in\mathbb{C}$ and $\nu(\varphi_A-(k-1)\log|z|,o)\geq1$. By Lemma \ref{lem1}, we have
$$\nu(\chi\circ\varphi-(k-1)\log|z|,o)=\nu(\varphi_A-(k-1)\log|z|,o),$$
which implies $\nu(\chi\circ\varphi,o)=\nu(\varphi_A,o)=C_1\geq k$.

By the assumption, we have $\lim\limits_{z\to0}(\chi\circ\varphi-C_0\varphi)=-\infty$. Then, we infer from $\varphi\leq C_2\log|z|+O(1)$ that $\lim\limits_{z\to0}(\chi\circ\varphi-C_0C_2\log|z|)=-\infty$. Hence,
$$\lim\limits_{z\to0}(e^{-2(\chi\circ\varphi-C_1\log|z|)}-e^{-2(\varphi_A-C_1\log|z|)})=+\infty,$$
It follows from $\nu(\chi\circ\varphi,o)=\nu(\varphi_A,o)=C_1\geq k$ that $|z|^{2(k-1)}(e^{-2\chi\circ\varphi}-e^{-2\varphi_A})$ is not locally integrable near $o$, which is a contradiction.
\end{proof}

\section{Proof of main results}

We are now in a position to prove our main results.\\
\textbf{\emph{Proof of Theorem} \ref{cluster}.}
Let $$\varphi(z)=\log|z_1|+\sum\limits_{k=1}^{\infty}\alpha_k\log(|z_1|+|\frac{z_2}{k}|^{\beta_k}),$$ where $\alpha_k=M^{-k}$, $\beta_k=M^{2k}$ and $M\geq2$ . Then $\varphi\in Psh(\mathbb{C}^n)$, $\mathscr{I}(\varphi)_o=(z_1)\cdot\mathcal{O}_n$ and $\nu(\varphi,(0,z_2))=1,\ \forall z_2\not=0$.

Given $\varepsilon\in(0,1]$. As $(z_1,z_2)\cdot\mathcal{O}_n\subset$rad$\mathscr{I}((1-\varepsilon)\varphi)_o$, there exists an integer $N>0$ such that $z_2^N\in\mathscr{I}((1-\varepsilon)\varphi)_o$.
Note that $\mathscr{I}((1-\varepsilon)\varphi)_o\subset\mathscr{I}((1-\varepsilon)\varphi_k)_o$ for every $k$, where $\varphi_k=\log|z_1|+\alpha_k\log(|z_1|+|\frac{z_2}{k}|^{\beta_k})$. To prove the desired result, it is sufficient to show that for any integer $m>0$, there exists $k_0$ and $\varepsilon_0>0$ such that $z_2^m\notin\mathscr{I}((1-\varepsilon)\varphi_{k_0})_o,\ \forall\varepsilon\in(0,\varepsilon_0]$.

Considering the integration of $|z_2|^{m}e^{-(1-\varepsilon)\varphi_k}$ over unit bidisk $\Delta\times\Delta$, we have
\begin{equation*}
\begin{split}
&\int_{\Delta\times\Delta}(|z_2|^{m}e^{-(1-\varepsilon)\varphi_k})^2d\lambda_{2}\\
=&\int_{\Delta^*}(\frac{1}{|z_1|^{(1-\varepsilon)(1+\alpha_k)}})^2\frac{\sqrt{-1}}{2}dz_1{\wedge}d\bar z_1
 \int_{\Delta}(\frac{|z_2|^m}{(1+\frac{|z_2/k|^{\beta_k}}{|z_1|})^{(1-\varepsilon)\alpha_k}})^2
 \frac{\sqrt{-1}}{2}dz_2{\wedge}d\bar z_2\\
=&\int_{\Delta^*}(\frac{1}{|z_1|^{(1-\varepsilon)(1+\alpha_k)}})^2\frac{\sqrt{-1}}{2}dz_1{\wedge}d\bar z_1
 \int_{\tilde\Delta}(\frac{|kz_1^{1/\beta_k}|^{m+1}\cdot|w|^m}{(1+|w|^{\beta_k})^{(1-\varepsilon)\alpha_k}})^2
 \frac{\sqrt{-1}}{2}dw{\wedge}d\bar w\\
\geq&\ C\cdot\int_{\Delta^{*}}(\frac{1}{|z_1|^{(1-\varepsilon)(1+\alpha_k)-\frac{m+1}{\beta_k}}})^2\frac
 {\sqrt{-1}}{2}dz_1{\wedge}d\bar z_1.
\end{split}
\end{equation*}
where $C>0$ is some constant and $\tilde\Delta$ is the domain of integration in the new variable. Then, for any $m$, it follows from $\alpha_k=M^{-k}$ and $\beta_k=M^{2k}$ that there exists $k_0$ and $\varepsilon_0>0$ such that $(1-\varepsilon_0)(1+\alpha_{k_0})-\frac{m+1}{\beta_{k_0}}\geq1$, which implies $z_2^m\notin\mathscr{I}((1-\varepsilon)\varphi_{k_0})_o,\ \forall\varepsilon\in(0,\varepsilon_0]$.
\hfill $\Box$

\noindent{\textbf{\emph{Proof of Theorem} \ref{equi_general}.}
Let $$\xi:\mathbb{C}^n\to\mathbb{C}^n,\ z\mapsto(z_1,z_1z_2,\dots,z_1z_n)$$ be a holomorphic mapping. Then $\xi$ is an isomorphism outside $\{z_1=0\}$. Hence, for sufficiently small polydisc $\Delta^n_r$, we have
$$\int_{\Delta^n_r}|z_1|^{2(n-1)}(e^{-2\chi(\varphi\circ\xi)}-e^{-2(\varphi_A\circ\xi)})d\lambda_n
  =\int\limits_{\xi(\Delta^n_r)\backslash\{z_1=0\}}(e^{-2\chi\circ\varphi}-e^{-2\varphi_A})d\lambda_n<\infty.$$
That is to say, for almost every $(z_2,...,z_n)\in\Delta^{n-1}_r$, we have
$$\int_{\Delta_r}|z_1|^{2(n-1)}(e^{-2\chi(\varphi\circ\xi)}-e^{-2(\varphi_A\circ\xi)})d\lambda_1<\infty.$$
It follows from Lemma \ref{lem2} that $\lim\limits_{t\to-\infty}|\chi(t)-C_0t|<\infty$.
\hfill $\Box$\\

\noindent{\textbf{\emph{Proof of Theorem} \ref{equisingular}.} It is enough to prove the dimension two case. In fact, if $e^{-2\varphi}-e^{-2\varphi_A}$ is locally integrable near the origin $o=(o',o'')$ for general $n\ (n\geq3)$ case, then for almost all $(a_3,...,a_n)\in\mathbb{C}^{n-2}$ near $o''$, $\big(e^{-2\varphi}-e^{-2\varphi_A}\big)|_{\{z_3=a_3,...,z_n=a_n\}}$ is locally integrable near $o'$, which is a contradiction.

Assume that $e^{-2\varphi}-e^{-2\varphi_A}$ is integrable on a relatively compact neighborhood $U_0$ of $o$. Then, $e^{-2c\varphi}-e^{-2c\varphi_A}$ is integrable on $U_0$ for any $0\leq c\leq1$. As $\mathscr{I}(\varphi)=(z_1)\cdot\mathcal{O}_{\mathbb{C}^2}$ on $U_0$, it follows that $e^{-2\varphi_A}$ is not locally integrable near $\{z_1=0\}|_{U_0}$. By Theorem \ref{Siu_decomposition}, on $U_0$, we have
$\varphi_A=\lambda\log|z_1|+\psi$, where $\lambda\geq1$ and $\psi$ is a plurisubharmonic function near $o$ such that $(\{z_1=0\},o)\not\subset(\{\nu(\psi,z)\geq c\},o),\ \forall c\in\mathbb{Q}^+$. It follows from Lemma \ref{Fubini} that $\psi|_{z_1=0}\not\equiv-\infty$.

If $\lambda>1$, then for every $z_2\not=0$, there exists $\varepsilon_0>0$ such that $\varphi_A\leq(1+\varepsilon_0)\log |z_1|$ near $(0,z_2)$. Hence, $e^{-2\frac{\varphi_A}{1+\varepsilon_0}}\geq e^{-2\log|z_1|}$ is not locally integrable near $(0,z_2)$, which is a contradiction to the local integrability of $e^{-2\frac{\varphi}{1+\varepsilon_0}}$ near $(0,z_2)$. Hence, we have $\varphi_A=\log|z_1|+\psi$ on $U_0$.

Since $\psi|_{z_1=0}\not\equiv-\infty$, then the Lelong number $\nu(\psi|_{z_1=0},o')<\infty$. Hence, we have an integer number $N>0$ such that $$z_2^N\in\mathscr{I}(\psi|_{z_1=0})_o\subset\mathscr{I}(c\psi|_{z_1=0})_o,\ \forall c\in(0,1).$$
By Theorem \ref{O-T}, there exists a holomorphic function $F_c$ on a Stein neighborhood $V_0{\subset}{\subset}U_0$ of $o$ such that
$$F_c(z)|_{V_0{\cap}\{z_1=0\}}=z_2^N|_{V_0{\cap}\{z_1=0\}}\ \mbox{and}\ F_c(z)_{o}\in\mathscr{I}(c\psi+c\log|z_1|)_{o}=\mathscr{I}(c\varphi_A)_{o}.$$

Let $F_c(z)=z_2^N+z_1\cdot G_c(z)$ for some holomorphic function $G_c(z)$ near $o$. By the assumption, we have
$$(z_1)\cdot\mathcal{O}_2=\mathscr{I}(\varphi)_o=\mathscr{I}(\varphi_A)_o\subset\mathscr{I}(c\varphi_A)_o.$$
It follows that $z_2^N=F_c(z)-z_1\cdot G_c(z)\in\mathscr{I}(c\varphi_A)_o$ for any $c\in(0,1)$. However, by the proof of Theorem \ref{cluster}, we know that there exists $\varepsilon_0>0$ such that $$z_2^N\notin\mathscr{I}((1-\varepsilon)\varphi)_o,\ \forall\varepsilon\in(0,\varepsilon_0].$$
which is a contradiction to the assumption, i.e., $e^{-2\varphi}-e^{-2\varphi_A}$ is not locally integrable near $o$.
\hfill $\Box$

\vspace{.1in} {\em Acknowledgements}. Both authors would like to sincerely thank our supervisor, Professor Xiangyu Zhou for valuable discussions and useful comments.

The second author is also grateful to Beijing International Center for Mathematical Research for its good working conditions.


\end{document}